\documentclass[12pt]{amsart}
\usepackage{amsfonts}
\usepackage{ifthen}
\usepackage{verbatim}
\usepackage{amsthm}
\usepackage{amsmath}
\usepackage{amssymb}
\usepackage{graphicx}
\usepackage{subfigure}
\usepackage{amscd,amssymb,amsthm}
\usepackage{color,xcolor}

\newcounter{minutes}
\divide\time by 60
\newcounter{hours}
\multiply\time by 60 \addtocounter{minutes}{-\time}

\title{Some refinements of Hermite-Hadamard inequality and an open problem }

\author{Slavko Simi\'c}

\address{ Mathematical Institute SANU, Kneza Mihaila 36, 11000
Belgrade, Serbia} \email{ ssimic@turing.mi.sanu.ac.rs}

\keywords{Hermite-Hadamard integral inequality, differentiable
function, convex function} \subjclass [2010]{26D07(26D15)}

\newtheorem{theorem}[equation]{Theorem}

\newtheorem{lemma}[equation]{Lemma}

\newtheorem{corollary}[equation]{Corollary}
\newtheorem{remark}[equation]{Remark}

\newcommand{\beq}{\begin{equation}}
\newcommand{\eeq}{\end{equation}}

\numberwithin{equation}{section}

\begin{document}

\def\thefootnote{}
\footnotetext{ \texttt{\tiny File:~\jobname .tex,
          printed: \number\year-\number\month-\number\day,
          \thehours.\ifnum\theminutes<10{0}\fi\theminutes}
} \makeatletter\def\thefootnote{\@arabic\c@footnote}\makeatother

\begin{abstract}
We presented here a refinement of Hermite-Hadamard inequality as a
linear combination of its end-points. The problem of best possible
constants is closely connected with well known Simpson's rule in
numerical integration. It is solved here for a wide class of
convex functions, but not in general. Some supplementary results
are also given.
\end{abstract}

\maketitle
\section{Introduction}
A function $f: I\subset\mathbb R\to \mathbb R$ is said to be
convex on an non-empty interval $I$ if the inequality
\begin{equation}\label{eq1}
f(px+qy)\le pf(x)+qf(y)
\end{equation}
holds for all $x,y\in I$ and all non-negative $p, q; p+q=1$.

If the inequality (\ref{eq1}) reverses, then $f$ is said to be
concave on $I$. \cite{hlp}

Let $f: I\subset\mathbb R\to \mathbb R$ be a convex function on an
interval $I$ and $a,b\in I$ with $a<b$. Then
\begin{equation}\label{eq2}
f(\frac{a+b}{2})\le\frac{1}{b-a}\int_a^b
f(t)dt\le\frac{f(a)+f(b)}{2}.
\end{equation}

This double inequality is well known in the literature as
Hermite-Hadamard integral inequality for convex functions. See,
for example, \cite{np} and references therein.

If $f$ is concave, both inequalities in (\ref{eq2}) hold in the
reversed direction.

Our task in this paper is to improve the inequality (\ref{eq2}) in
a simple manner, i.e., to find some positive constants $\alpha,
\beta,\gamma,\delta$ such that the relations
\begin{equation}\label{eq3}
\gamma(f(a)+f(b))+\delta f(\frac{a+b}{2})\le\frac{1}{b-a}\int_a^b
f(t)dt\le \alpha(f(a)+f(b))+\beta f(\frac{a+b}{2}),
\end{equation}

hold for any convex $f$.

Taking $f(t)=Ct, C\in\mathbb R/\{0\}$, it can be easily seen that
both conditions
\begin{equation}\label{eq4}
2\alpha+\beta=1; 2\gamma+\delta=1,
\end{equation}

are necessary for (\ref{eq3}) to hold.

Denote

$$
M(\gamma, \delta)=M_f(a,b;\gamma,\delta):=
\gamma(f(a)+f(b))+\delta f(\frac{a+b}{2}),
$$

and

$$
N(\alpha, \beta)=N_f(a,b;\alpha,\beta):= \alpha(f(a)+f(b))+\beta
f(\frac{a+b}{2}).
$$

Since

$$
N(\alpha, \beta)=(2\alpha)(\frac{f(a)+f(b)}{2})+ \beta
f(\frac{a+b}{2})
$$
$$
\le \max\{\frac{f(a)+f(b)}{2},
f(\frac{a+b}{2})\}=\frac{f(a)+f(b)}{2},
$$

and, consequently

$$
M(\gamma, \delta)=(2\gamma)(\frac{f(a)+f(b)}{2})+ \delta
f(\frac{a+b}{2})
$$
$$
\ge \min\{\frac{f(a)+f(b)}{2},
f(\frac{a+b}{2})\}=f(\frac{a+b}{2}),
$$

it follows that the inequality (\ref{eq3}) represents a refinement
of Hermite-Hadamard inequality (\ref{eq2}).

Now, it can be seen that the bound $M(0,1)$ is best possible in
general case. Indeed, let $\gamma\in (0, 1/2]$ be fixed and the
relation

$$
M_f(0,1;\gamma, \delta)\le \int_0^1 f(t)dt
$$

holds for arbitrary convex $f$.

Then the convex function $f(t)=t^{1/\gamma}$ gives a
counter-example.

This means that the left-hand side of Hermite-Hadamard inequality
cannot be improved, in general, by the form of (\ref{eq3}).

Nevertheless, such improvement is possible for some special
classes of convex functions (see Corollary \ref{col1} below).

The case of the bound $N(\alpha,\beta)$ is significantly harder.
We found the value $N(1/4, 1/2)$ for which the right-hand side of
(\ref{eq3}) holds for any integrable convex function. Since
$N(\alpha, \beta)$ is monotone increasing in $\alpha$, because

$$
\frac{d}{d\alpha}N_f(a,b;\alpha,
\beta)=f(a)+f(b)-2f(\frac{a+b}{2})\ge 0,
$$

it follows that the right-hand side of (\ref{eq3}) also holds for
 $\alpha\in [1/4,1/2]$.

In search for best possible constants, note that

$$
\alpha\ge\frac{\frac{1}{b-a}\int_a^b
f(t)dt-f(\frac{a+b}{2})}{f(a)+f(b)-2f(\frac{a+b}{2})}:=F_f(a,b)
$$

and, if $f\in C^{\infty}(I)$ and $f''(t)>0$ for $t\in I$, then
$\lim_{b\to a}F_f(a,b)=1/6$, independently of $f$.

Therefore we obtain a potentially best possible bound $N(1/6,
2/3)$. Unfortunately, as is shown in Theorem \ref{thm3}, this
bound is best possible only for the class of differentiable convex
functions for which $f^{(4)}(t)\ge 0, t\in I$.

Hence, we can formulate the following

{\bf Open Question} {\it Find the best possible bound $N(\alpha^*,
\beta^*)$ valid for all convex mappings from the class
$C^{\infty}(I)$}.

Since the function $g(t)=\frac{4t^3\sqrt t}{35}-\frac{t^4}{12}$,
convex on $I=[0,1]$, gives the value $\alpha=.18128>1/6$, we have
that $\alpha^*\in [.18128, .25]$.

\section{Results and proofs}

We shall begin with the basic contribution to the problem defined
above.

\begin{theorem}\label{thm1} Let $f: I\subset\mathbb R\to \mathbb R$ be a convex function on an
interval $I$ and $a,b\in I$ with $a<b$. Then
$$
\frac{1}{b-a}\int_a^b
f(t)dt\le\frac{1}{4}(f(a)+f(b))+\frac{1}{2}f(\frac{a+b}{2}):=N(1/4,1/2).
$$
\end{theorem}

\begin{proof} We shall derive the proof by Hermite-Hadamard
inequality itself. Indeed, applying twice the right part of this
inequality, we get

$$
\frac{2}{b-a}\int_a^{\frac{a+b}{2}}
f(t)dt\le\frac{1}{2}(f(a)+f(\frac{a+b}{2})),
$$

and

$$
\frac{2}{b-a}\int_{\frac{a+b}{2}}^b
f(t)dt\le\frac{1}{2}(f(\frac{a+b}{2})+f(b)).
$$

Summing, the result appears. Therefore, HH inequality has a
self-improving property.
\end{proof}

For the sake of further refinements, we shall consider in the
sequel functions from the class $C^{(4)}(I)$ i.e., functions which
are continuously differentiable up to fourth order on an interval
$I\subset\mathbb R$.

We give firstly a sharp improvement of the result from Theorem
\ref{thm1}.

\begin{theorem}\label{thm2} Let $f\in C^{(4)}(I)$ be convex on $I$ together
with its second derivative. Then for each $a,b\in I, a<b$,

$$
\frac{(b-a)^2}{48}f''(\frac{a+b}{2})\le
N(1/4,1/2)-\frac{1}{b-a}\int_a^b f(t)dt\le
\frac{(b-a)^2}{96}[f''(a)+f''(b)].
$$

 If $f$ is convex and $f''$ concave on $I$, then

 $$
\frac{(b-a)^2}{96}[f''(a)+f''(b)]\le
N(1/4,1/2)-\frac{1}{b-a}\int_a^b f(t)dt\le
\frac{(b-a)^2}{48}f''(\frac{a+b}{2}).
 $$
\end{theorem}

\begin{proof} We need the following two assertions.
\begin{lemma}\label{l1} \cite{s} If $h$ is convex on $I=[a,b]$ and, for $x,y\in I,
x+y=a+b$, then

$$
2h(\frac{a+b}{2})\le h(x)+h(y)\le h(a)+h(b).
$$
\end{lemma}

\begin{remark} Note that this result is a pre-HH inequality, i.e., HH
inequality is its direct consequence. Indeed, let $x=pa+qb,
y=qa+pb$ for $p,q\ge 0, p+q=1$. Then $x,y\in I$ and $x+y=a+b$.
Hence,
$$
2h(\frac{a+b}{2})\le h(pa+qb)+h(qa+pb)\le h(a)+h(b).
$$

Integrating this expression over $p\in [0,1]$ we obtain the HH
inequality.
\end{remark}

\begin{lemma}\label{l2} Let $f\in C^{(4)}(I)$ and $a,b\in I, a<b$.Then the following identity holds.
$$
N(1/4,1/2)-\frac{1}{b-a}\int_a^b f(t)dt=
\frac{(b-a)^2}{16}\int_0^1 t(1-t)[f''(x)+f''(y)]dt,
$$

with $x:=a\frac{t}{2}+b(1-\frac{t}{2}),
y:=b\frac{t}{2}+a(1-\frac{t}{2})$.
\end{lemma}

It is not difficult to prove this identity  by double partial
integration of its right-hand side.

Since $x+y=a+b$ and $f''$ is convex/concave, applying Lemma
\ref{l1} the proof readily follows.

\end{proof}

Another improvement of HH inequality is given in the next

\begin{theorem} \label{thm3} Let $f\in C^{(4)}(I)$ and $a,b\in I,
a<b$. If $f, f''$ are convex on $I$, then

$$
\frac{1}{b-a}\int_a^b f(t)dt\le N(1/6,2/3),
$$

and the coefficients $1/6, 2/3$ are best possible for this class
of functions.

If $f''$ is concave on $I$ then the reversed inequality takes
place.
\end{theorem}

\begin{proof} Note that the coefficients $1/6$ and $2/3$ are
involved in well-known Simpson's rule which is of importance in
numerical integration. It says that

\begin{lemma} \label{l3} \cite{u} For an integrable $f$, we have

$$
\int_{x_1}^{x_3} f(t)
dt=\frac{1}{3}h(f_1+4f_2+f_3)-\frac{1}{90}h^5 f^{(4)}(\xi),
 (x_1<\xi<x_3),
$$

where $h:=x_2-x_1=x_3-x_2$.
\end{lemma}

Now, taking $x_1=a, x_2=(a+b)/2, x_3=b$, we get $h=(b-a)/2$. Also,
convexity/concavity of $f''$ on $I$ implies that
$f^{(4)}(\xi)\gtrless 0$ and the proof follows.
\end{proof}

Combining the second part of this theorem with the result of
Theorem \ref{thm1}, we get

\begin{corollary}\label{col1} For $f\in C^{(4)}(I)$ let $f$ be convex
and $f''$ concave on $I$. Then

$$
N(1/6,2/3)\le \frac{1}{b-a}\int_a^b f(t)dt\le N(1/4,1/2),
$$

which gives a proper answer, regarding this class of functions, to
the problem posted in Introduction, .
\end{corollary}

Further refinement of the assertion from Theorem \ref{thm3} is
possible.

\begin{theorem} \label{thm4} For $f\in C^{(4)}(I)$ let $f$
and $f''$ be convex on $I$. Then

$$
0\le \frac{1}{6}[f(a)+f(b)]+\frac{2}{3}
f(\frac{a+b}{2})-\frac{1}{b-a}\int_a^b f(t)dt
$$
$$
\le\frac{(b-a)^2}{324} [f''(a)+f''(b)-2f''(\frac{a+b}{2})].
$$

If $f$ is convex and $f''$ concave on $I$, then

$$
0\le \frac{1}{b-a}\int_a^b f(t)dt-\frac{1}{6}[f(a)+f(b)+ 4
f(\frac{a+b}{2})]
$$
$$
\le \frac{(b-a)^2}{324} [2f''(\frac{a+b}{2})-(f''(a)+f''(b))].
$$
\end{theorem}

The above theorem tightly refines Simpson's rule for this class of
functions.

\begin{proof} The left part is proved in Theorem \ref{thm3}. For
the right part we shall use an integral identity

\begin{lemma}\label{l4}
$$
N(1/6,2/3)- \frac{1}{b-a}\int_a^b
f(t)dt=\frac{(b-a)^2}{48}\int_0^1 t(2-3t)[f''(x)+f''(y)]dt,
$$

where $x$ and $y$ are the same as in Lemma \ref{l2}.
\end{lemma}

Writing,

$$
\int_0^1 t(2-3t)[\cdot]dt = \int_0^{2/3}
t(2-3t)[\cdot]dt-\int_{2/3}^1 t(3t-2)[\cdot]dt,
$$

and applying Lemma \ref{l1} to each integral separately, the
result appears since

$$
\int_0^{2/3} t(2-3t)dt=\int_{2/3}^1 t(3t-2)dt=\frac{4}{27}.
$$

\end{proof}

\begin{remark}\label{rem1} Note that the convexity condition on
$f$ in last three theorems is superfluous. It is stated there just
to keep the connection with Hermite-Hadamard inequality.
\end{remark}

\section{Applications in Means theory}

A {\em mean} is a map $M: \mathbb R_+\times\mathbb R_+\to \mathbb
R_+$, with a property

$$
\min\{a,b\}\le M(a,b)\le\max\{a,b\},
$$

for each $a,b\in \mathbb R_+$.

Hence $M$ is necessary reflexive, $M(a,a)=a$.

Most known ordered family of means is the following family
$\Delta$ of elementary means,

$$
\Delta: \  H\le G \le L\le I\le A\le S,
$$
where
$$
H=H(a, b):=2(1/a+1/b)^{-1}; \ \ G=G(a, b):=\sqrt {ab}; \ \ L=L(a,
b):=\frac{b-a}{\log b-\log a};
$$
$$
 I=I(a, b):=\frac{1}{e}(b^b/a^a)^{1/(b-a)}; \ \ A=A(a,
 b):=\frac{a+b}{2};
 S=S(a,b):=a^{\frac{a}{a+b}}b^{\frac{b}{a+b}},
$$
are the harmonic, geometric, logarithmic, identric, arithmetic
 and Gini mean, respectively.

 As an illustration of our results, we shall give in the sequel
 some sharp approximations of logarithmic and identric means.

 \begin{theorem}\label{thm5} The inequality $G\le L\le A$ can be
 improved to

 $$
 \frac{1}{3}(A+2G)-\frac{2}{81}(\frac{A-G}{L})^2(A+G)\le
 L\le\frac{1}{3}(A+2G).
 $$

 Similarly, an approximation of $1/L$ in terms of the arithmetic
 and harmonic means is given by

 $$
 \frac{A-H}{6A^2}\le\frac{1}{2}(\frac{1}{A}+\frac{1}{H})-\frac{1}{L}
 \le\frac{A(A-H)}{6H^2}(\frac{4}{H}-\frac{3}{A}).
 $$

 \end{theorem}

 \begin{proof} Applying Theorem \ref{thm4} with $f=e^t$, we obtain

 $$
 0\le \frac{1}{6}(e^x+e^y)+\frac{2}{3}e^{\frac{x+y}{2}}-\frac{e^x-e^y}{x-y}
 $$
 $$
 \le \frac{(x-y)^2}{324}(e^x+e^y-2e^{\frac{x+y}{2}}).
 $$

 Since $x$ and $y$ are arbitrary real numbers, putting $x=\log b, y=\log
 a$, we get

 $$
 0\le \frac{1}{3}(A+2G)-L\le\frac{(\log b-\log a)^2}{162}(A-G)
 $$
 $$
 =\frac{4}{162}(\frac{\log b-\log
 a}{b-a})^2(A^2-G^2)(A+G)=\frac{2}{81}(\frac{A-G}{L})^2(A+G),
 $$

 and the proof is done.

 For the second part, applying Theorem \ref{thm2} with $f=1/t, f''=2/t^3$, we
 get

 $$
 \frac{(b-a)^2}{24}\frac{1}{A^3}\le\frac{1}{4}(\frac{1}{a}+\frac{1}{b})+\frac{1}{2A}
 -\frac{1}{L}\le\frac{(b-a)^2}{48}(\frac{1}{a^3}+\frac{1}{b^3}).
 $$

 Now, the identities $1/a+1/b=2/H, (b-a)^2=4A(A-H), AH=G^2$ yields
 the proof.

 \end{proof}

 Finally, we shall give some interesting inequalities for the
 identric mean.

 \begin{theorem} \label{thm6} For arbitrary positive $a,b$ we have

 $$
 A^{2/3}G^{1/3}\le I\le
 A^{2/3}G^{1/3}\exp\bigl(\frac{(A-H)^2}{162H}(\frac{1}{A}+\frac{2}{H})\bigr);
 $$

 $$
 A^{4/3}(a,b)
 S^{2/3}(a,b)\exp\bigl(-\frac{4}{81}\frac{(A(a,b)-H(a,b))^2}{A(a,b)H(a,b)}\bigr)\le
 I(a^2,b^2)\le A^{4/3}(a,b)S^{2/3}(a,b).
 $$
 \end{theorem}

 \begin{proof} Applying Theorem \ref{thm4} with $f=-\log t$, we
 obtain the proof.

 For the second part, for $f=t\log t$,
 we get

 $$
 \frac{1}{b-a}\int_a^b f(t)dt=\frac{1}{4}(\frac{b^2\log
 b^2-a^2\log a^2}{b-a}-(a+b))=\frac{a+b}{4}\log I(a^2,b^2).
 $$

 Since $f''=1/t$, Theorem \ref{thm4} yields

 $$
 \frac{1}{6}(a\log a+b\log b)+\frac{2}{3}A\log
 A-\frac{(b-a)^2}{324}(\frac{1}{a}+\frac{1}{b}-\frac{2}{A})
 $$
 $$
 \le \frac{a+b}{4}\log I(a^2,b^2)\le \frac{1}{6}(a\log a+b\log b)+\frac{2}{3}A\log
 A,
 $$

 and the proof follows by dividing the last expression with $a+b=2A$.
 \end{proof}


\begin{thebibliography}{HIMPS}
\footnotesize



\bibitem[HLP]{hlp}{\sc G.H. Hardy, J.E. Littlewood and G. Polya}, {\em
    Inequalities}, Cambridge University Press, Cambridge, 1978.



\bibitem[NP]{np} {\sc
C. P. Nikulesku and L. E. Persson,} Old and new on the
Hermite-Hadamard inequality, Real Analysis Exchange, Vol. 29(2)
(2003/4) pp. 663-685.

\bibitem[S]{s} {\sc S. Simic,} On a convexity property, Krag. J.
Math. Vol. 40(2) (2016) pp. 166-171.

\bibitem[U]{u}{\sc C. W. Ueberhuber}, {\em Numerical Computation 2}, Berlin, Springer-Verlag, 1997.

\end{thebibliography}
\end{document}